\documentclass[a4paper]{amsart}

\usepackage{lmodern,bm}              
\usepackage{longtable}
\usepackage{tikz}
\usepackage[all]{xy}

\CompileMatrices

\newtheorem{theorem}{Theorem}[section]

\newtheorem{lemma}[theorem]{Lemma}
\newtheorem{proposition}[theorem]{Proposition}
\newtheorem{corollary}[theorem]{Corollary}

\theoremstyle{definition}

\newtheorem{construction}[theorem]{Construction}

\newtheorem{remark}[theorem]{Remark}

\numberwithin{equation}{theorem}

\def\vector2#1#2{\left(\begin{array}{c} #1 \\ #2 \end{array}\right)}
\def\Cl{{\rm Cl}}

\def\CC{{\mathbb C}}
\def\KK{{\mathbb K}}
\def\TT{{\mathbb T}}
\def\ZZ{{\mathbb Z}}

\def\QQ{{\mathbb Q}}
\def\PP{{\mathbb P}}

\def\grad{{\rm grad}}

\def\Chi{{\mathbb X}}

\def\quot{/\!\!/}

\def\bangle#1{{\langle #1 \rangle}}
\def\rq#1{\widehat{#1}}
\def\t#1{\widetilde{#1}}
\def\b#1{\overline{#1}}

\def\SL{{\rm SL}}

\def\Spec{{\rm Spec}}

\def\cone{{\rm cone}}

\def\trop{{\rm trop}}
\def\lin{{\rm lin}}

\title[Non-complete rational $T$-varieties of complexity one]%
{Non-complete rational $T$-varieties  \\ of complexity one}

\author[J.~Hausen and M.Wrobel]{J\"urgen~Hausen and Milena Wrobel}

 \address{Mathematisches Institut, Universit\"at T\"ubingen,
Auf der Morgenstelle 10, 72076 T\"ubingen, Germany}
\email{juergen.hausen@uni-tuebingen.de}

\address{Mathematisches Institut, Universit\"at T\"ubingen,
Auf der Morgenstelle 10, 72076 T\"ubingen, Germany}
\email{milena.wrobel@math.uni-tuebingen.de}

\subjclass[2010]{14L30, 13A05, 13F15}

\begin{document}

\maketitle

\begin{abstract}
We consider rational varieties with a torus 
action of complexity one and extend the 
combinatorial approach via the Cox ring 
developed for the complete case in earlier 
work to the non-complete, e.g. affine, case. 
This includes in particular a description of
all factorially graded affine algebras 
of complexity one with only constant 
homogeneous invertible elements in terms of
canonical generators and relations. 
\end{abstract}


\section{The main results}
\label{sec:mainres}

The purpose of this note is to extend  
the toolkit developed in~\cite{HaSu,HaHeSu,HaHe}
for normal complete rational varietes with a torus 
action of complexity one also to non-complete 
varieties, for example affine ones. 
Recall that the \emph{complexity} of a variety~$X$ 
with an effective action of an algebraic torus~$T$
equals $\dim(X)-\dim(T)$.
The case of complexity zero is the theory of 
toric varieties which due to their complete 
combinatorial encoding form a very popular 
field of research.

Complexity one is a natural step beyond the toric 
case showing new phenomena as for example new 
classes of singularities.
However, due to the fact that the generic quotient 
is merely a curve, still much of the geometry is encoded in 
combinatorial properties of the torus action; 
see~\cite{KKMSD, Ti, AlHa} for general 
approaches based on this observation.
If we restrict to \emph{rational} varieties
with torus action of complexity one, 
then we have a finitely generated \emph{Cox ring} 
$$ 
\mathcal{R}(X) 
\ = \ 
\bigoplus_{\Cl(X)} \Gamma(X,\mathcal{O}(D)).
$$
If, in addition, $X$ is complete, then~\cite{HaSu} 
provides an explicit description of the 
Cox ring in terms of generators and specific 
trinomial relations.
Applying the machinery around Cox rings to this 
particular setting leads to an elementary
and very concrete approach in the complete case; 
see~\cite{ArDeHaLa} for the basics 
and~\cite{ArHaHeLi, BeHaHuNi} for some applications.
As indicated, we extend this approach to the 
non-complete case.

We work over an algebraically closed field $\KK$ 
of characteristic zero.
The basic task is to describe all  
Cox rings of normal rational varieties $X$ with 
an effective torus action $T \times X \to X$ 
of complexity one also in the non-complete case. 
As it is needed for the uniqueness of 
the Cox ring, we require  
$\Gamma(X,\mathcal{O}^*) = \KK^*$.
From the algebraic point of view, a Cox ring is 
firstly a finitely generated integral $\KK$-algebra~$R$
graded by a finitely generated abelian group~$K$
such that that there are only constant invertible 
homogeneous elements.
The most important characterizing property of 
a Cox ring is then \emph{$K$-factoriality}, which 
means that we have unique factorization in the 
multiplicative monoid of non-zero homogeneous elements. 
Observe that for a torsion-free grading group,
$K$-factoriality is equivalent to the usual unique 
factorization but in general it is definitely 
weaker.

In a first step, we describe all finitely 
generated integral $\KK$-algebras~$R$ that admit 
an effective factorial $K$-grading of complexity 
one, where effective means that the $w \in K$ 
with $R_w \ne 0$ generate $K$ as a group and 
complexitiy one means that the rational vector
space $K \otimes \QQ$ is of dimension one less
than~$R$ is.
The first results in this direction concern 
the case  $R_0 = \KK$ in dimension two,
see~\cite{Mo,Ku}.
The case $R_0 = \KK$ in arbitrary dimension 
was settled in~\cite{HaHe} 
and occurs as part of \emph{Type~2} in our 
subsequent considerations. 
A simple example of \emph{Type~1} presented
below is the coordinate algebra of the special 
linear group~$\SL(2)$:
$$ 
R 
\ = \ 
\KK[T_1,T_2,T_3,T_4] / \bangle{T_1T_2+T_3T_4-1},
\qquad
Q 
\ = \ 
\left[
\begin{array}{rrrr}
1 & -1 & 0 & 0 
\\
0 & 0 & 1 & -1 
\end{array}
\right],
$$
where the matrix $Q$ specifies a $\ZZ^2$-grading 
of $R$ by assigning to the variable $T_i$ the $i$-th 
column of $Q$ as its degree; note that this reflects 
multiplication of diagonal torus elements $s$ from 
the left and $t$ from the right modulo trivially 
acting pairs~$s,t$.
Here comes the general construction.

\begin{construction}
\label{constr:RAP0}
Fix integers $r,n > 0$, $m \ge 0$ and a partition 
$n = n_\iota+\ldots+n_r$ starting at $\iota \in \{0,1\}$.
For each $\iota \le i \le r$, fix a tuple
$l_{i} \in \ZZ_{> 0}^{n_{i}}$ and define a monomial
$$
T_{i}^{l_{i}}
\ := \
T_{i1}^{l_{i1}} \cdots T_{in_{i}}^{l_{in_{i}}}
\ \in \
\KK[T_{ij},S_{k}; \ \iota \le i \le r, \ 1 \le j \le n_{i}, \ 1 \le k \le m].
$$
We will also write $\KK[T_{ij},S_{k}]$ for the 
above polynomial ring. 
We distinguish two settings for the input data 
$A$ and $P_0$ of the graded $\KK$-algebra $R(A,P_0)$.

\medskip
\noindent
\emph{Type 1.}
Take $\iota = 1$.
Let $A := (a_1, \ldots, a_r)$ be a list of pairwise 
different elements of~$\KK$.
Set $I := \left\{1, \ldots, r-1\right\}$
and define for every $i \in I$ a polynomial
$$
g_{i} 
\ := \ 
T_i^{l_i} - T_{i+1}^{l_{i+1}} - (a_{i+1}-a_i) 
\ \in \ 
\KK[T_{ij}, S_k].
$$
We build up an $r \times (n+m)$ matrix 
from the exponent vectors $l_1, \ldots, l_r$ of these 
polynomials:
$$
P_{0}
\ := \
\left[
\begin{array}{cccccc}
l_{1} &  & 0 & 0  &  \ldots & 0
\\
\vdots  & \ddots & \vdots & \vdots &  & \vdots
\\
 0 &  & l_{r} & 0  &  \ldots & 0
\end{array}
\right].
$$

\medskip
\noindent
\emph{Type 2.}
Take $\iota = 0$.
Let $A:= (a_0, \ldots, a_r)$ be a $2 \times (r+1)$-matrix with pairwise 
linearly independent columns $a_i \in \KK^2$. 
Set $I := \left\{0, \ldots, r-2\right\}$ and for every $i \in I$
define
$$
g_{i}
\ :=  \
\det
\left[
\begin{array}{lll}
T_i^{l_i} & T_{i+1}^{l_{i+1}} & T_{i+2}^{l_{i+2}}
\\
a_i & a_{i+1}& a_{i+2}
\end{array}
\right]
\ \in \
\KK[T_{ij},S_{k}].
$$
We build up an $r \times (n+m)$ matrix 
from the exponent vectors $l_0, \ldots, l_r$ of these 
polynomials:
$$
P_{0}
\ := \
\left[
\begin{array}{ccccccc}
-l_{0} & l_{1} &  & 0 & 0  &  \ldots & 0
\\
\vdots & \vdots & \ddots & \vdots & \vdots &  & \vdots
\\
-l_{0} & 0 &  & l_{r} & 0  &  \ldots & 0
\end{array}
\right].
$$
We now define the ring $R(A,P_0)$ simultaneously 
for both types in terms of the data $A$ and $P_0$.
Denote by $P_0^*$ the transpose of $P_0$ and consider 
the projection
$$
Q \colon \ZZ^{n+m} 
\ \to \ 
K_{0} 
\ := \ 
\ZZ^{n+m}/\mathrm{im}(P_{0}^{*}).
$$
Denote by $e_{ij},e_{k} \in \ZZ^{n+m}$ the canonical
basis vectors corresponding to the variables 
$T_{ij}$, $S_{k}$.
Define a $K_0$-grading on $\KK[T_{ij},S_{k}]$ 
by setting
$$
\deg(T_{ij}) \ := \ Q(e_{ij}) \ \in \ K_{0},
\qquad
\deg(S_{k}) \ := \ Q(e_{k}) \ \in \ K_{0}.
$$
This is the coarsest possible grading of
$\KK[T_{ij},S_{k}]$ leaving the variables 
and the $g_i$ homogeneous.
In particular, we have  a $K_{0}$-graded 
factor algebra
$$
R(A,P_{0})
\ := \
\KK[T_{ij},S_{k}] / \bangle{g_{i}; \ i \in I}.
$$
\end{construction}

We gather the basic properties of the graded 
algebras just constructed; 
the corresponding proofs are given in 
Section~\ref{sec:sec2}.
Below, we mean by a \emph{$K_0$-prime}
a homogeneous non-zero non-unit which,
whenever it divides a product of 
homogeneous elements, it also divides one of
the factors.

\begin{theorem}
\label{Theorem1}
Let $R(A,P_0)$ be a $K_0$-graded ring as provided by 
Construction~\ref{constr:RAP0}.
\begin{enumerate}
\item
The ring $R(A,P_0)$ is an integral, normal 
complete intersection of dimension $n+m-r+1$.
\item
The $K_0$-grading on $R(A,P_0)$ is effective, 
factorial of complexity one and $R(A,P_0)$
has only constant invertible homogeneous 
elements. 
\item
The variables $T_{ij}$ and $S_k$ define pairwise 
nonassociated $K_0$-prime generators for $R(A,P_0)$.
\item
Suppose $r\geq 2$ and $n_il_{ij}>1$ for all $i, j$. 
Then the ring $R(A,P_0)$ is factorial if and only if $K_0$
is torsion free. Moreover, 
\begin{enumerate}
\item
in case of Type~1, $R(A,P_0)$ is factorial if 
and only if one has
$\gcd(l_{i1}, \ldots, l_{in_i}) = 1$ 
for $i = 1, \ldots, r$,
\item
in case of Type~2, $R(A,P_0)$ is factorial if 
and only if
any two of $l_i := \gcd(l_{i1}, \ldots, l_{in_i})$
are coprime.
\end{enumerate}
\end{enumerate}
\end{theorem}

Observe that the situation of~(iv) can always be 
achieved by eliminating the variables that occur
in a linear term of some relation.
The following result shows that 
Construction~\ref{constr:RAP0} yields in fact all
affine algebras with property~(ii) of the above 
theorem; see Section~\ref{sec:sec2} for the proof.

\begin{theorem}
\label{Theorem2}
Let $K$ be a finitely generated abelian group and
$R$ a finitely generated, integral, normal 
$\KK$-algebra with an effective, factorial 
$K$-grading of complexity one and only constant 
invertible homogeneous elements. 
Then $R$ is isomorphic to a $K_0$-graded 
$\KK$-algebra $R(A,P_0)$ as provided by 
Construction~\ref{constr:RAP0}.
\end{theorem}

We turn to Cox rings of rational varieties with 
a torus action of complexity one. 
They will be obtained as suitable downgradings 
of the algebras $R(A,P_0)$ of Construction~\ref{constr:RAP0}.
Here comes the precise recipe.

\begin{construction}
\label{constr:RAPdown}
Let integers $r$, $n = n_\iota + \ldots + n_r$, $m$ 
and data $A$ and $P_0$ of Type~1 or Type~2 as 
in Construction~\ref{constr:RAP0}. 
Fix $1 \le s \le n + m - r$, choose an integral
$s \times (n + m)$ matrix $d$ and build the 
$(r+s) \times (n + m)$ stack matrix
$$
P 
\ := \
\left[
\begin{array}{c}
P_0
\\
d
\end{array}
\right].
$$
We require the columns of $P$ to be pairwise 
different primitive vectors generating
$\QQ^{r+s}$ as a vector space. 
Let $P^*$ denote the transpose of $P$ and 
consider the projection
$$
Q \colon 
\ZZ^{n+m} 
\ \to \ 
K 
\ := \ 
\ZZ^{n+m} / \mathrm{im}(P^*).
$$
Denoting as before by $e_{ij}, e_k \in \ZZ^{n+m}$ the 
canonical basis vectors corresponding to
the variables $T_{ij}$ and $S_k$, we obtain a 
$K$-grading on $\KK[T_{ij}, S_k]$ by setting
$$
\deg(T_{ij}) \ := \ Q(e_{ij} ) \ \in \ K,
\qquad\qquad
\deg(S_k) \ := \ Q(e_k) \ \in \ K.
$$
This $K$-grading coarsens the $K_0$-grading of 
$\KK[T_{ij},S_k ]$ given in Construction~\ref{constr:RAP0}. 
In particular, we have the $K$-graded factor algebra
$$
R(A,P)
\ := \
\KK[T_{ij},S_{k}] / \bangle{g_{i}; \ i \in I}.
$$
\end{construction}

We present the basic properties of this 
construction; see Section~\ref{sec:sec2} 
for the proof.
Below, we say that the grading is 
\emph{almost free} if removing any 
variable $T_{ij}$ or $S_k$ still leaves 
enough variables such that their degree generate
the grading group.

\goodbreak

\begin{theorem}
\label{Theorem3}
Let $R(A,P)$ be a $K$-graded ring as provided by 
Construction~\ref{constr:RAPdown}. 
\begin{enumerate}
\item
The $K$-grading on $R(A,P)$ is almost free, 
factorial, and $R(A,P)$
has only constant invertible homogeneous elements.
\item
The variables $T_{ij}$ and $S_k$ define pairwise different
nonassociated $K$-prime generators for $R(A,P)$.
\end{enumerate}
\end{theorem}

Knowledge of the Cox ring allows to (re)construct 
the underlying varieties. 
As in the complete case, we will obtain  
$A_2$-varieties, i.e.~admitting an embedding 
into a toric variety.
This comprises in particular the affine and,
more generally, the quasiprojective case.
We make use of the language of bunched rings,
see~\cite{ArDeHaLa} for an introduction.

\begin{construction}
\label{constr:RAPandPhi}
Let $R(A,P)$ be a $K$-graded ring as provided by 
Construction~\ref{constr:RAPdown} and 
$\mathfrak{F} = (T_{ij},S_k)$ the 
canonical system of generators.
Consider
$$
H \ := \ \Spec \, \KK[K],
\qquad\qquad
\b{X}(A,P) \ :=  \ \Spec \, R(A,P),
$$
Then $H$ is a quasitorus and the $K$-grading
of $R(A,P)$ defines an action of $H$ on~$\b{X}$.
Any true $\mathfrak{F}$-bunch $\Phi$ defines 
an $H$-invariant open set and a good quotient
$$ 
\rq{X}(A,P,\Phi) \ \subseteq \ \b{X}(A,P),
\qquad\qquad
X(A,P,\Phi) \ = \ \rq{X}(A,P,\Phi) \quot H.
$$
The action of $H_0 = \Spec \, \KK[K_0]$ leaves 
$\rq{X}(A,P,\Phi)$ invariant and induces an 
action of the torus $T = \Spec \, \KK[\ZZ^s]$ 
on $X(A,P,\Phi)$.
\end{construction}

From~\cite[Thm.~3.2.1.4]{ArDeHaLa} and the 
observation that the presence of a torus action 
of complexity one implies rationality in case 
of a finitely generated divisor class group,
we infer the basic properties of the above 
construction.

\begin{theorem}
\label{Cor:1}
Consider a $T$-variety $X = X(A,P,\Phi)$ as 
provided by Construction~\ref{constr:RAPandPhi}.
Then $X$ is a normal, rational $A_2$-variety
with only constant invertible functions
and the action of $T$ on $X$ is of complexity one.
Dimension, divisor class group and Cox ring
of $X$ are given by
$$ 
\dim(X) \ = \ s+1,
\qquad 
\Cl(X) \ \cong \ K,
\qquad
\mathcal{R}(X) \ \cong \ R(A,P).
$$ 
\end{theorem}

The following converse is proven in 
Section~\ref{sec:sec2} 
and concerns \emph{$A_2$-maximal} varieties 
that means $A_2$-varieties that can not 
be realized as an open subset with 
nonempty complement of codimension at 
least two in another $A_2$-variety; 
this setting includes in particular the affine and, 
more generally, the semiprojective case.

\begin{theorem} 
\label{Cor:2}
Let $X$ be an irreducible, normal, rational, 
$A_2$-maximal variety 
with only constant invertible functions, 
finitely generated divisor class group
and a torus action of complexity one.
Then $X$ is equivariantly isomorphic to 
a variety $X(A,P,\Phi)$ provided by 
Construction~\ref{constr:RAPandPhi}.
\end{theorem}

In the case of \emph{affine}, normal,
rational varieties 
with a torus action of complexity one.
Here, the whole machinery 
boils down to the following statement.

\begin{corollary}
\label{Cor:3}
Let $X$ be an irreducible, normal, rational 
affine variety with 
only constant invertible functions, 
finitely generated divisor class group 
and an effective algebraic torus action
of complexity one. 
Then $X$ is equivariantly isomorphic 
to a variety $\Spec \, R(A,P)_0$ acted 
on by the torus $H_0/H$,
where $R(A,P)$ is as in Construction~\ref{constr:RAPdown}
and the columns of $P$ generate the extremal rays 
of a pointed cone in~$\QQ^{r+s}$. 
\end{corollary}

Finally, in Section~\ref{sec:sec-ex},
we illustrate our methods by discussing 
the well-known case of normal affine 
$\KK^*$-surfaces~\cite{FiKp,FlZa}.
We take a closer look at du Val 
singularities and show how their Cox rings 
and resolutions are obtained using our 
framework; see~\cite{FaGoLa, LiSu, MDB} 
for earlier treatments based on other methods.


\section{Proofs of the main results}
\label{sec:sec2}

We first show that the algebras provided by 
Constructions~\ref{constr:RAP0} and~\ref{constr:RAPdown} 
have indeed the desired properties:
the assertions of Theorem \ref{Theorem1} are verified
in Propositions~\ref{prop:hominv2const} to~\ref{prop:typ1fact}
and then Theorem~\ref{Theorem3} is proven.
We work in the notation of 
Constructions~\ref{constr:RAP0} and~\ref{constr:RAPdown}.

\begin{proposition}
\label{prop:hominv2const}
Let $R(A,P_0)$ be a $\KK$-algebra of Type~1 as in 
Construction~\ref{constr:RAP0}.
Then every $K_0$-homogeneous invertible element of $R(A, P_0)$ 
is constant.
\end{proposition}

\begin{lemma}
\label{LemmGl}
Notation as for Type 1 in Construction~\ref{constr:RAP0}.
For any two indices $1 \le i,j \le r$, set 
$$
g_{ij} \ := \ T_i^{l_i} - T_j^{l_j} + a_i - a_j.
$$
For any three $1 \le i,j,k \le r$, we have 
$g_{ij} =  g_{ik}  - g_{jk}$
and $G := \{g_{1 r}, \ldots, g_{r-1 \, r}\}$ 
is a reduced Gr\"obner basis with respect to 
the lexicographical ordering
for $\bangle{g_1, \ldots, g_{r-1}}$.
\end{lemma}

\begin{proof}
The identities among the $g_{ij}$ are obvious.
Since $g_i = g_{i\, i+1}$ holds, we see that $G$ 
generates $\bangle{g_1, \ldots, g_{r-1}}$.
With $\alpha_{ij} := a_j - a_i$, the $S$-polynomials 
of $G$ are of the form
$$
T_i^{l_i}  T_r^{l_r} -T_j^{l_j}  T_r^{l_r} +T_i^{l_i} 
\alpha_{jr} -  T_j^{l_j}  \alpha_{ir}
\ = \
g_{ir} (T_r^{l_r} + \alpha_{jr}) - g_{jr}(T_r^{l_r} + \alpha_{ir}).
$$
In particular, they all reduce to zero with 
respect to $G$ and thus $G$ is the desired 
Gr\"obner basis for $\bangle{g_1, \ldots, g_{r-1}}$.
Obviously $G$ is reduced.
\end{proof}

\begin{proof}[Proof of Proposition~\ref{prop:hominv2const}]
Let $f \in \KK[T_{ij},S_k]$ define a $K_0$-homogeneous unit
in $R(A,P_0)$ with inverse defined by $g \in \KK[T_{ij},S_k]$.
We first show that $f$ and hence $g$ is of $K_0$-degree 
zero.
We have a presentation
$$
fg - 1 
\ = \ 
\sum_{i = 1}^{r-1} h_i g_i, \qquad h_i \in \KK[T_{ij},S_k].
$$
Suppose that $f$ is of nonzero $K_0$-degree.
Then $g$ is so and the constant term of 
$fg -1$ equals $-1$. 
Thus, at least one of $h_i$ must have 
a nonzero constant term and we may 
rewrite the presentation as
$$
fg - 1 
\ = \ 
\sum_{i = 1}^{r-1} \t{h_i} g_i + \beta_i g_i
\ = \ 
\sum_{i = 1}^{r-1} \t{h_i} g_i 
+ \beta_i(T_i^{l_i} - T_{i+1}^{l_{i+1}} -\alpha_{i \, i+1}),
$$
where the $h_i\in \KK[T_{ij},S_k]$ have constant term zero
and at least one $\beta_i$ is nonzero. 
Adding 1 to the left and the right hand side gives
$$
fg 
\ = \ 
\sum_{i = 1}^{r-1} \t{h_i} g_i + \beta_i(T_i^{l_i} - T_{i+1}^{l_{i+1}}).
$$
By Lemma~\ref{LemmGl}, at least two 
different $T_j^{l_j},T_k^{l_k}$
are not cancelled in the right hand side.
Consider monomials $f_j,f_k$ of $f$ 
dividing $T_j^{l_j},T_k^{l_k}$ respectively.
Since $f$ is $K_0$-homogeneous, the exponents of $f_j$ and $f_k$ 
differ by an element of the row lattice of~$P_0$.
This works only for $f_j=1$ or $f_j = T_j^{l_j}$.
We conclude that $f$ and hence $g$ is of $K_0$-degree 
zero; a contradiction.

Having seen that $f$ and $g$ are of $K_0$-degree zero,
we conclude that they are polynomials in the $T_i^{l_i}$.
Using the structure of the $g_{ir}$ we may bring the 
representatives $f$ and $g$ in the form 
$$
f 
\ = \ 
\sum_i \lambda_i (T_r^{l_r})^i, 
\qquad\qquad 
g \ = \ \sum_j \kappa_j (T_r^{l_r})^j.
$$
Then also $fg -1$ is a polynomial in $T_r^{l_r}$. 
Since $fg - 1$ belongs to $\bangle{g_{1 \, r}, \ldots, g_{r-1 \, r}}$,
its reduction by the set $G$ of Lemma~\ref{LemmGl}
equals zero. 
This means $fg-1 = 0$ and thus $f,g \in \KK^*$.
\end{proof}

\begin{proposition}
\label{prop:RAP0integral}
Let $R(A,P_0)$ be a $\KK$-algebra of Type~1 as in 
Construction~\ref{constr:RAP0}.
Then $R(A,P_0)$ is an integral, regular complete 
intersection of dimension $n+m-r+1$. 
Moreover the $K_0$-grading of $R(A,P_0)$ is 
effective and of complexity one.
\end{proposition}

\begin{lemma}
\label{LemmIrr}
Let $G$ be a quasitorus and $X$ a (normal) 
affine $G$-variety with only constant invertible 
homogeneous functions. 
Then $X$ is connected (irreducible).
\end{lemma}

\begin{proof}
Consider the induced action of $G$ on the set 
$Y = \{X_1, \ldots, X_r\}$ of connected components.
Then $Y$ is a single $G$-orbit, because otherwise 
we can write $X$ is a union of disjoint open 
$G$-invariant sets which in turn yields 
nonconstant invertible 
functions on $X$.
The stabilizer $G_1 \subseteq G$ of $X_1 \in Y$
is a closed subgroup and we have the 
homomorphism $\pi \colon G \to G/G_1$.
Write $X_i = g_i \cdot X_1$ with suitable 
$g_1, \ldots, g_r \in G$.
Then, for every character $\chi$ on $G/G_1$, 
we obtain an invertible regular function
$f_\chi$ on $X$ sending $x \in g_i \cdot X_1$ 
to $\chi(\pi(g_i))$.
By construction, $f_\chi$ is homogeneous
with respect to $\chi$.
Thus every $f_\chi$ is constant, which means
$G = G_1$ and thus $X = X_1$.
\end{proof}

\begin{proof}[Proof of Proposition~\ref{prop:RAP0integral}]
Consider $\b{X} := V(g_1,\ldots, g_{r-1}) \subseteq \KK^{m+n}$.
We first show that for every $z \in \b{X}$ 
the Jacobian of $g_1,\ldots, g_{r-1}$ 
is of full rank.
The Jacobian is of the form $(J_g,0)$ with
$$
J_g 
\ := \ 
\begin{bmatrix}
 \delta_{1 \ 1} & \delta_{1\ 2} & 0 && & \cdots && & 0\\
0 & \delta_{2\ 2} & \delta_{1\ 3} & 0  &&& & & \\
\vdots & & & & \vdots&&& & \vdots\\
 & && & &0 &  \delta_{2\ r-2} & \delta_{1\ r-1} & 0\\
0 & &\cdots & && &0  & \delta_{2\ r-1} & \delta_{1\ r} \\
\end{bmatrix}
$$
where each $\delta_{ti}$ is a nonzero multiple 
of $\delta_i:= \grad \ T_i^{l_i}$. 
Let $z \in \KK^{m+n}$ be any point with $J_g(z)$ 
not of full rank. 
Then $\delta_i(z) = \delta_j(z) = 0$ for some 
$i \neq j$. 
This implies ${z_{ik} = 0 =z_{jl}}$ for some 
$0 \leq k  \leq n_i$, $0 \leq l \leq n_j$. 
It follows $T_i^{l_i}(z) = T_j^{l_j}(z) = 0$ and 
thus $z \notin \overline{X}$. 
So the Jacobian is of full rank for any 
$z \in \overline{X}$.

We conclude that $g_1,\ldots, g_{r-1}$ generate
the vanishing ideal of $\b{X}$ and that 
$\b{X} = \Spec \, R(A,P_0)$ is smooth. 
Lemma~\ref{LemmIrr} yields that 
$\b{X}$ is connected
and, by smoothness, irreducible.
Thus $R(A,P_0)$ is integral.
Moreover the dimension of $\overline{X}$ 
and hence $R(A,P_0)$ is $n+m-(r-1)$ 
and thus $R(A,P_0)$ is a complete intersection.
\end{proof}

\begin{proposition}
\label{prop:VarK0prime}
Let $R(A,P_0)$ be of Type~1. 
Then the variables $T_{ij}, S_k$ define 
pairwise nonassociated $K_0$-prime 
elements in $R(A,P_0)$. 
If furthermore the ring $R(A,P_0)$ is 
factorial and $n_il_{ij}>1$ holds, 
then $T_{ij}$ is even prime.
\end{proposition}

\begin{proof}
First observe that, by the nature of 
relations, any two different variables 
define a zero set of codimension at least 
two in $\b{X}$.
Thus the variables are pairwise 
non-associated.
Since $R(A,P_0)$ is integral and 
$R(A,P_0) \cong R(A,P_0')[S_1,\ldots, S_m]$ 
holds with~$P_0'$ obtained from $P_0$ 
by deleting the zero columns,
the $S_k$ are even prime. 

We now turn to the $T_{ij}$ and exemplarily 
treat $T_{11}$.
The task is to show that the divisor 
of $T_{11}$ in $\b{X}$ is $H_0$-prime
that means that its prime components have 
multiplicity one and are transitively permuted 
by $H_0$. 
First we claim
$$
V(\b{X};  T_{11})
\ := \ 
V(T_{11}) \cap \b{X}
\ = \ 
\b{H_0 \cdot z}
\ \subseteq \
\KK^{n+m}.
$$
Indeed, the zero set of $T_{11}$ in $\overline{X}$ 
is given by the equations 
$$
T_{11} \ = \ 0, 
\qquad 
T_s^{l_s} \ = \ a_s-a_1, 
\qquad 
2 \leq s \leq r.
$$
Set 
$h := T_{12} \cdots T_{rn_r} \cdot S_1 \cdots S_m$ 
and let $z \in \KK^{n+m}_h$ be a point 
satisfying the above equations. 
Then $z$ is of the form 
$(0,z_{12}, \ldots, z_{rn_r}, z_1, \ldots, z_m)$ 
with nonzero $z_{ij}$ and $z_k$ 
and any other such point $z' \in \KK_h^{n+m}$ 
is given as
$$
z' 
= 
t \cdot z 
= 
(0,t_{12}z_{12},\ldots,t_{rn_r}z_{rn_r},t_1z_1, \ldots, t_mz_m), 
\quad
t \in (\KK^*)^{n+m}, 
\quad
t_s^{l_s}=1.
$$
This means $t \in H_0$ and $V(\b{X}_h; T_{11}) = H_0 \cdot z$. 
Since the common zero set of any two different variables
is of codimension at least two in $\overline{X}$,
our claim follows.
In particular, $H_0$ permutes transitively the 
components of the divisor defined by 
$T_{11}$ on $\b{X}$.
To obtain only multiplicities one, observe 
that the Jacobian of the above equations is of 
full rank at any point of 
$H_0 \cdot z \subseteq V(\b{X}; T_{11})$. 

To prove the supplement, one proceeds by exactly the 
same arguments as used in the proof 
of~\cite[Prop.~3.4.2.8(ii)]{ArDeHaLa}.
\end{proof}

\begin{proposition}
Let $R(A,P_0)$ be of Type~1. 
Then $R(A,P_0)$ is $K_0$-factorial.
\end{proposition}

\begin{proof}
First observe that the quasitorus 
$H_0 \cong \Spec \, \KK[K_0]$ 
equals the kernel of the homomorphism of tori
$$
\varphi\colon \TT^{n+m} \ \to \ \TT^r,
\qquad 
(t_{ij}, t_k) \ \mapsto \ (t_1^{l_1}, \ldots, t_r^{l_r}).
$$
Denote the coordinates of $\TT^r$ by $U_1,\ldots, U_r$.
Then the relations $g_i$ are pullbacks of the affine linear 
forms
$$
g_i 
\ = \ 
\varphi^*(h_i), 
\qquad 
h_i 
\ := \ 
U_i - U_{i+1} - (a_{i+1}-a_i) 
\ \in \ 
\KK[U_1^{\pm}, \ldots, U_r^{\pm 1}].
$$
The $h_i$ generate the vanishing ideal of an $r$ times 
punctured affine line in $\TT^r$ and thus
$$
(R(A,P_0)_t)_0
\ = \ 
\KK[U_1^{\pm}, \ldots, U_r^{\pm 1}] 
/ 
\bangle{h_1, \ldots, h_{r-1}}
$$ 
is a factorial ring, where $t$ is the product over 
all the variables $T_{ij}$ and $S_k$.
Now Proposition~\ref{prop:VarK0prime}
and~\cite[Cor.~3.4.1.6]{ArDeHaLa} tell us
that $R(A,P_0)$ is $K_0$-factorial.
\end{proof}

\begin{proposition}\label{prop:Gcd}
Let $R(A,P_0)$ be of Type~1. 
Then the variable $T_{ij}$ is prime
in $R(A,P_0)$ if and only 
if $1 = \gcd(l_{k1}, \ldots ,l_{kn_k})$ 
holds for all $k \neq i$.
\end{proposition}

\begin{proof}
We treat exemplarily $T_{11}$. 
By Lemma~\ref{LemmGl}, the ideal 
of relations of $R(A,P_0)$ is 
generated by $g_{1 2}, \ldots, g_{1 r}$.
Thus $T_{11}$ generates a prime ideal 
if and only if the the following ideal
is prime 
$$
\bangle{T_j^{l_j} +a_j -a_1; \ j \neq 1}
\  \subseteq \ 
\KK[T_{ij}; \ (i,j) \neq (1,1)].
$$
This is equivalent to the statement that 
$(l_2, 0, \ldots, 0),\ldots, (0, \ldots, 0, l_r)$ 
generate a primitive sublattice of $\ZZ^{n-n_1}$. 
This in turn holds if and only if $l_{k1}, \ldots ,l_{kn_k}$ 
have greatest common divisor one for all $k \neq 1$.
\end{proof}

\begin{proposition}
\label{prop:typ1fact}
Let $R(A,P_0)$ be of Type~1 and suppose 
that $r\geq 2$ and $n_il_{ij}>1$ hold 
for all $i, j$. 
Then the following statements are equivalent.
\begin{enumerate}
\item
The ring $R(A,P_0)$ is factorial.
\item
The group $K_0$ is torsion free.
\item
We have $\gcd(l_{i1}, \ldots, l_{in_i}) = 1$
for $i = 1, \ldots, r$.
\end{enumerate}
\end{proposition}

\begin{proof}
Let $K_0$ be torsion free. 
Then $K_0$-factoriality implies factoriality 
of $R(A,P_0)$, see~\cite[Thm.~3.4.1.11]{ArDeHaLa}. 
If $R(A,P_0)$ is factorial, then 
Proposition~\ref{prop:VarK0prime} says 
that the generators $T_{ij}$ are 
prime.
This implies  
$\gcd(l_{k1}, \ldots ,l_{kn_k})  = 1$ for all 
$k$, see Proposition~\ref{prop:Gcd}. 
If the latter holds, then the rows of $P_0$ 
generate a primitive sublattice of $\ZZ^{n+m}$ 
and thus $K_0$ is torsion free.
\end{proof}

\begin{proof}[Proof of Theorem~\ref{Theorem3}]
We first show that every $K$-homogeneous 
unit $f \in R(A,P)_w$ is constant.
For this, it suffices 
to show that $f$ is $K_0$-homogeneous, 
see Proposition~\ref{prop:hominv2const}.
From~\cite[Rem.~3.4.3.2]{ArDeHaLa} we infer 
that the downgrading map $K_0 \to K$ has 
kernel~$\ZZ^s$.
Consider the inverse $g \in R(A,P)_{-w}$ 
of $f \in R(A,P)_w$
and the decompositions into $K_0$-homogeneous 
parts 
$$
f \ = \ \sum f_i, \quad f_i \in R(A,P_0)_{u_i}, 
\qquad\qquad
g \ = \ \sum g_j, \quad g_j \in R(A,P_0)_{v_j},
$$
where $u_i = w_0 + u_i'$ with $u_i' \in \ZZ^s$ 
and $v_j = -w_0 + v_j'$ with $v_j' \in \ZZ^s$
for some fixed $w_0 \in K_0$ projecting to $w \in K$;
we identify the kernel of $K_0 \rightarrow K$ with
$\ZZ^s$. 
Let $f_{i_0}, f_{i_1}, g_{j_0}, g_{j_1}$ denote the 
terms, where $u_{i_0}'$, $v_{j_0}'$ are minimal 
and $u_{i_1}'$, $v_{j_1}'$ are maximal with 
respect to the lexicographical ordering 
on $\ZZ^s$.
As $1$ is of $K_0$-degree zero, we obtain
$$
0 
\ = \ 
\deg(f_{i_0} g_{j_0}) 
\ = \ 
w_0+u_{i_0}' - w_0 + v_{j_0}' 
\ = \ 
u_{i_0}' + v_{j_0}'.
$$
and analogously $u_{i_1}' + v_{j_1}' = 0$. 
We conclude $u_{i_0}' = u_{i_1}'$ and $v_{j_0}' = v_{j_1}'$.
Consequently, $f$ is homogeneous with respect to the 
$K_0$-grading.

Using~\cite[Lemma~2.1.4.1]{ArDeHaLa},
we see that the $K$-grading of $R(A,P)$ 
is almost free.
By~\cite[Lemma~3.4.3.5]{ArDeHaLa}, the 
variables $T_{ij}$ and $S_k$ define pairwise 
nonassociated $K$-primes in $R(A,P)$.
Finally, \cite[Thms.~3.4.1.5, 3.4.1.11 and Cor.~3.4.1.6]{ArDeHaLa}
show that the $K$-grading of $R(A,P)$ 
is factorial.
\end{proof}

Now we turn to the converse statements.
For this, we adapt the ideas of~\cite{HaSu}
to our more general setting.

\begin{proof}[Proof of Theorem~\ref{Theorem2}]
Consider $X = \Spec \, R$ with the action of 
$H := \Spec \; \KK[K]$ defined by the grading.
We follow the lines of~\cite[Sec.~4.4.2]{ArDeHaLa}.
Denote by $E_1,\ldots,E_m$ be the prime divisors 
on $X$ such that for any $x \in E_k$ the isotropy 
group $H_x$ is infinite and consider the 
$H$-invariant open subset
$$ 
X_0 
\ := \
\{x \in X; \; H_x \text{ is finite}\}
\ \subseteq \ 
X. 
$$
Then there is a geometric quotient $X_0 \to X_0/H$
with a possibly non-separated smooth curve $X_0/H$.
Consider the separation $X_0/H \to Y$ and let 
$a_\iota, \ldots, a_r \in Y$ be points such that 
every fiber of $X_0/H \to Y$ comprising more than 
one point lies over some $a_i$ and every prime 
divisor of $X$ with non-trivial general 
$H$-isotropy lies over some $a_i$; we denote these
prime divisors by $D_{ij}$, where the $i$ indicates
that $D_{ij}$ lies over $a_i$.

According to~\cite[Thm.~4.4.2.1]{ArDeHaLa},
this quotient
is the characteristic space over the possibly 
non-separated curve $X_0/H$ and 
we have a canonical well-defined 
pullback isomorphism of $K$-graded algebras
$$ 
\mathcal{R}(X_0/H)[T_{ij},S_k]
/ \bangle{T_{ij}^{l_{ij}} - 1_{z_{ij}}}
\ \to \ 
\Gamma(X,\mathcal{O}),
$$
where $S_k$ and $T_{ij}$ are sent to functions with 
divisor $E_k$ and $D_{ij}$ respectively and 
$1_{z_{ij}}$ is the pullback of the canonical 
section of a point $z_{ij} \in X_0/H$ lying 
over $y_i \in Y$. 
As $X_0/H$ is smooth, has only constant 
invertible global functions and finitely generated 
divisor class group, we end up with $Y$ being  
either the affine or the projective line.
The Cox ring of $X_0/H$ is given as 
$$ 
\mathcal{R}(X_0/H)
\ = \ 
\mathcal{R}(Y)[U_{ij}] / \bangle{U_{i1} \cdots U_{in_i} - 1_{a_i}},
$$
where the $U_{ij}$ represent the canonical sections
of the points $z_{i1}, \ldots, z_{in_i} \in X_0/H$ 
lying over $a_i \in Y$ and $1_{a_i}$ is the pullback of the 
canonical section of $a_i$ with respect to 
$X_0/H \to Y$, see~\cite[Prop.~4.4.3.4]{ArDeHaLa}. 
Now, if $Y = \KK$ holds, we set $\imath := 1$ and 
represent $\mathcal{R}(Y)$ as 
$$
\mathcal{R}(Y) 
\ = \ 
\KK[V_1,\ldots, V_r] / \bangle{V_i - V_{i+1} - (a_{i+1}-a_i)}. 
$$ 
Plugging this into the above descriptions of $\mathcal{R}(X_0/H)$
and $\Gamma(X,\mathcal{O})$ gives us Type~1 of 
Construction~\ref{constr:RAP0}.
If $Y = \PP_1$ holds, then we set $\imath := 0$, replace 
the $a_i \in Y$ with representatives 
$a_i \in \KK^2 \setminus \{0\}$ and obtain
$$ 
\mathcal{R}(Y) 
\ = \ 
\KK[V_0,\ldots, V_r] / \bangle{g_0, \ldots, g_{r-2}},
\qquad
h_i \ := \ 
\det
\left[
\begin{array}{lll}
V_i & V_{i+1} & V_{i+2} 
\\
a_i & a_{i+1}& a_{i+2}
\end{array}
\right].
$$
Combining this description with the above presentations 
of $\mathcal{R}(X_0/H)$ and $\Gamma(X,\mathcal{O})$ 
leads to Type~2 of Construction~\ref{constr:RAP0}.

So far, we verified that the algebra 
$R = \Gamma(X,\mathcal{O})$ 
has the desired generators and relations. 
The generators are homogeneous with respect to
$K = \Chi(H)$.
As the $K_0$-grading of $R(A,P_0)$
is the finest possible with this 
property, we obtain a downgrading map
$K_0 \to K$.
Using the arguments of the proof 
of~\cite[Thm.~4.4.2.2]{ArDeHaLa}, we see 
that $K_0 \to K$ is an isomorphism.
\end{proof}

\begin{proof}[Proof of Theorem~\ref{Cor:2}]
One follows exactly the proof 
of~\cite[Thm.~4.4.1.6]{ArDeHaLa}, 
but uses our more general 
Theorem~\ref{Theorem3}
instead of~\cite[Thm.~4.4.2.2]{ArDeHaLa}.
\end{proof}

\begin{proof}[Proof of Corollary~\ref{Cor:3}]
Theorem~\ref{Cor:2} tells us $X \cong X(A,P,\Phi)$
as in Construction~\ref{constr:RAPandPhi}.
Since $X$ is affine, the open subset $\rq{X}(A,P,\Phi)$ 
equals the total coordinate space $\b{X}(A,P)$.
The latter means that $\Phi$ contains the trivial 
cone $\{0\}$. 
This is equivalent to saying that the columns of 
$P$ generate 
the extremal rays of a pointed cone in $\QQ^{r+s}$. 
\end{proof}


\section{Example: affine $\KK^*$-surfaces}
\label{sec:sec-ex}

To illustrate our methods, we consider 
the well-known case of normal affine 
$\KK^*$-surfaces~\cite{FiKp,FlZa} and 
take a closer look to those with at 
most du Val singularities.
As a preparation, 
we first adapt the construction 
of a canonical toric ambient variety 
from~\cite[Sec.~3.2.5]{ArDeHaLa} 
and then extend the resolution 
of singularities~\cite[Thm.~3.4.4.9]{ArDeHaLa}
to our setting.

\begin{construction}
Let $X = X(A,P,\Phi)$ be obtained from 
Construction~\ref{constr:RAPandPhi}.
The~\emph{tropical variety} of $X$ is 
the fan $\trop(X)$ in $\QQ^{r+s}$
having the maximal cones
$$ 
\lambda_i 
\ := \ 
\cone(v_{i1}) + \lin(e_{r+1}, \ldots, e_{r+s}),
\qquad
i \ = \ \iota, \ldots, r,
$$
where $v_{ij} \in \ZZ^{r+s}$ denote the first $n$ 
columns of $P$ and $e_k \in \ZZ^{r+s}$ the $k$-th 
canonical basis vector.

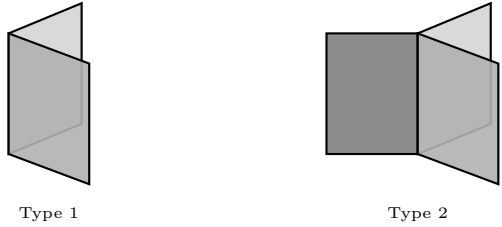
\begin{center}

\ \hfill
\begin{tikzpicture}[scale=0.4]
\draw[thick, draw=black, fill=gray!30!] (0,2) -- (2.4,3) -- (2.4,-1) -- (0,-2) -- cycle; 
\draw[thick, draw=black, fill=gray!60!, fill opacity=0.90] (0,2) -- (2.65,1) -- (2.65,-3) -- (0,-2) -- cycle; 
\node at (1.325,-4) {\tiny Type~1};
\end{tikzpicture}
\hfill 
\begin{tikzpicture}[scale=0.4]
\draw[thick, draw=black, fill=gray!90!] (0,2) -- (-3,2) -- (-3,-2) -- (0,-2) -- cycle; 
\draw[thick, draw=black, fill=gray!30!] (0,2) -- (2.4,3) -- (2.4,-1) -- (0,-2) -- cycle; 
\draw[thick, draw=black, fill=gray!60!, fill opacity=0.90] (0,2) -- (2.65,1) -- (2.65,-3) -- (0,-2) -- cycle; 
\node at (0,-4) {\tiny Type~2};
\end{tikzpicture}
\hfill \

\end{center}

\noindent 
For a face $\delta_0 \preceq \delta$ 
of the orthant $\delta \subseteq \QQ^{n+m}$,
let $\delta_0^* \preceq \delta$ denote the 
complementary face
and call $\delta_0$ \emph{admissible} if
\begin{itemize}
\item
the relative interior of $P(\delta_0)$ intersects $\trop(X)$, 
\item
the image $Q(\delta_0^*)$ comprises a cone of $\Phi$,
\end{itemize}
where  
$Q \colon \ZZ^{n+m} \to K = \ZZ^{n+m}/P^*(\ZZ^{r+s})$ 
is the projection. 
Then we obtain fans $\rq{\Sigma}$ in $\ZZ^{n+m}$ 
and~$\Sigma$ in $\ZZ^{r+s}$ of pointed cones 
by setting
$$ 
\rq{\Sigma}
\ := \ 
\{
\delta_1 \preceq \delta_0; \; 
\delta_0 \preceq \delta \text{ admissible} 
\},
\qquad
\Sigma 
\ := \ 
\{
\sigma \preceq P(\delta_0); \; 
\delta_0 \preceq \delta \text{ admissible} 
\}.
$$
With the toric varieties $\rq{Z}$ and $Z$ 
associated to $\rq{\Sigma}$ and $\Sigma$
respectively, 
we obtain a commutative diagramm of 
characteristic spaces
$$ 
\xymatrix{
{\rq{X}(A,P,\Phi)}
\ar@{}[r]|{\qquad\subseteq}
\ar[d]_{\quot H}
&
{\rq{Z}}
\ar[d]^{\quot H}
\\
X(A,P,\Phi)
\ar@{}[r]|{\qquad\subseteq}
&
Z
}
$$
The inclusions are $T$-equivariant closed 
embeddings, where $T$ acts on $Z$ 
as the subtorus of the $(r+s)$-torus 
corresponding to $0 \times \ZZ^{s} \subseteq  \ZZ^{r+s}$.
Moreover, $X(A,P,\Phi)$ intersects every closed toric 
orbit of $Z$.
\end{construction}

Observe that rays of the fan $\Sigma$ have precisely 
the columns of the matrix $P$ as its primitive 
generators.  
The following recipe for resolving singularities
directly generalizes~\cite[Thm.~3.4.4.9]{ArDeHaLa};
a related approach using polyhedral divisors is 
presented in~\cite{LiSu}.

\begin{construction}
\label{constr:ressing}
Let $X = X(A,P,\Phi)$ be obtained from 
Construction~\ref{constr:RAPandPhi}
and consider the canonical toric embedding
$X \subseteq Z$ and the defining fan $\Sigma$ 
of $Z$.
\begin{itemize}
\item
Let $\Sigma'  = \Sigma \sqcap \trop(X)$ be the 
coarsest common refinement.
\item
Let $\Sigma''$ be any regular subdivision of the 
fan $\Sigma'$.
\end{itemize}
Then $\Sigma'' \to \Sigma$ defines a proper toric 
morphism $Z'' \to Z$ and with the proper transform 
$X'' \subseteq Z''$ of $X \subseteq Z$, the morphism
$X'' \to X$ is a resolution of singularities.
\end{construction}

\begin{remark}
In the setting of Construction~\ref{constr:ressing},
the variety $X''$ has again a torus action 
of complexity one and thus is of the form 
$X'' = X(A'',P'',\Phi'')$.
We have $A'' = A$ and $P''$ is obtained from $P$ 
by inserting the primitive generators of $\Sigma''$ 
as new columns.
Moreover, $\Phi''$ is the Gale dual of $\Sigma''$, 
that means that with the corresponding projection $Q''$ 
and orthant $\delta''$ we have
$$ 
\Phi'' 
\ = \ 
\{
Q''(\delta_0^*); \; \delta_0 \preceq \delta''; 
\; P''(\delta_0) \in \Sigma''
\}.
$$
\end{remark}

As an example class we now consider rational normal 
affine  $\KK^*$-surfaces.
By Corollary~\ref{Cor:3}, 
any \emph{affine} rational normal variety $X$
with only constant invertible functions and a 
torus action of complexity one is of the form 
$$
X 
\ = \ 
X(A,P)
\ := \ 
\Spec \, R(A,P)_0.
$$ 
Moreover, using the fact that the columns of $P$ 
generated the extremal rays of a pointed cone in 
$\QQ^{r+s}$ we directly obtain the following.

\begin{remark}
Consider a rational normal affine $\KK^*$-surface 
$X = X(A,P)$. 
Then we have $s=1$ and there are three possible cases 
for the defining matrix $P$:
\begin{enumerate}
\item
The \emph{elliptic} case: we are in Type~2 and we have 
$n_0 = \ldots = n_r = 1$, $m=0$ and $s=1$.
\item
The \emph{parabolic} case: we are in Type~1 and we have
$n_1 = \ldots = n_r = 1$, $m=1$ and $s=1$.  
\item
The \emph{hyperbolic} case: we are in Type~1 and we have
$n_1, \ldots, n_r \le 2$, $m=0$ and $s=1$.  
\end{enumerate}
\end{remark}

We take a closer look at the surfaces $X = X(A,P)$ 
with at most du Val singularities. 
Recall that these are exactly the singularities 
with a resolution graph of type $\bm{A}$, $\bm{D}$ or~$\bm{E}$.
The singularities of type~$\bm{A}$ are precisely
the toric du Val surface singularities; 
we refer to~\cite{CoLiSc}
for an exhaustive treatment of this case.

\begin{proposition}
Let $X$ be a parabolic or hyperbolic normal 
affine $\KK^*$-surface. 
If $x_0 \in X$ is a du Val singularity, then 
it is of type $\bm{A}$.
\end{proposition}

\begin{proof}
In the parabolic and hyperbolic cases, the 
fan $\Sigma$ of the canonical ambient toric variety
is supported on the tropical variety $\trop(X)$
and we have $\Sigma' = \Sigma$ in the first step of
the resolution of singularities according to 
Construction~\ref{constr:ressing}.
The second step means regular subdivision of 
the purely two-dimensional fan $\Sigma' = \Sigma$
and, in the du Val case, we end up with resolution 
graphs of type $\bm{A}$; use~\cite[Sec.~5.4.2]{ArDeHaLa}
for computing intersection numbers.
\end{proof}

We turn to the elliptic case.
In case of a singularity of type $\bm{D}$ or $\bm{E}$,
we determine the possible $X$ 
and present the defining data and the Cox ring 
for $X$ as well as for the minimal 
resolution~$\tilde X$;
see~\cite{FaGoLa, LiSu, MDB} for other 
approaches.

\begin{proposition}
\label{thm:duval}
Let $X$ be an elliptic normal affine $\KK^*$-surface
with a du Val singularity $x_0 \in X$.
If $x_0$ is of type $\bm{A}$, then $X$ is 
an affine toric surface.
If $x_0$ is of type $\bm{D}$ or $\bm{E}$,
then $X \cong X(A,P)$, where 
$$ 
A
\ = \ 
{\tiny
\left[
\begin{array}{rrr}
0 & -1 & 1 
\\
1 & -1 & 0
\end{array}
\right]
},
$$
the defining matrix $P$ depends on the 
type of $x_0$ as shown in the table below;
we additionally present a defining 
equation for $X \subseteq \KK^3$ from~\cite{Ma}
and the relation $g$ of the Cox 
ring~$\mathcal{R}(X) = \KK[T_1,\ldots,T_4]/\bangle{g}$:
\begin{center}
{\small
\begin{longtable}{cccc}
$x_0$ &  equation in $\KK^3$ & matrix~$P$ & relation~$g$
\\
\hline
\\
$\bm{D}_q$ & \tiny{$T_1^2+T_2T_3^2+T_2^{q-1}$} &
\tiny{$\left[
\begin{array}{rrr}
 -2 & q-2& 0\\
 -2 & 0 & 2\\
 -1 & 1 & 1\\
\end{array}
\right]$}
& \tiny{$T_1^2 +T_2^ {q-2} + T_3 ^{2}$}
\\
\\
\hline
\\
$\bm{E}_6$ & \tiny{$T_1^2+T_2^3+T_3^{4}$}&
$\tiny{\left[
\begin{array}{rrr}
 -3 & 3& 0\\
 -3 & 0 & 2\\
 -2 & 1 & 1\\
\end{array}
\right]}$
& \tiny{$T_1^3 +T_2^3 + T_3 ^2$}
\\
\\
\hline
\\
$\bm{E}_7$&\tiny{$T_1^2+T_2^3+T_2T_3^3$}&
$\tiny{\left[
\begin{array}{rrr}
 -4 & 3& 0\\
 -4 & 0 & 2\\
 -3 & 1 & 1\\
\end{array}
\right]}$
& \tiny{$T_1^4 +T_2^3 + T_3 ^{2}$}
\\
\\
\hline
\\
$\bm{E}_8$&\tiny{$T_1^2+T_2^3+T_3^{5}$}&
$\tiny{\left[
\begin{array}{rrr}
 -5 & 3 & 0\\
 -5 & 0 & 2\\
 -4 & 1 & 1\\
\end{array}
\right]}$
& \tiny{$T_1^5 +T_2^3 + T_3 ^{2}$}
\end{longtable}
}
\end{center}
\noindent
Moreover, Construction~\ref{constr:ressing}
provides a minimal resolution 
of singularities $\tilde X \to X$ 
with $\tilde X = X(A,\tilde P, \tilde \Phi)$,
where~$\tilde P$ depends on the type 
of $x_0$ as shown below; we moreover 
list the relation $\tilde g$ of the Cox 
ring~$\mathcal{R}(\tilde X) = \KK[T_{ij},S_1]/\bangle{\tilde g}$:
\begin{center}
{\small
\begin{longtable}{cccc}
$x_0$ & matrix $\tilde P$ & relation $\tilde g$
\\
\hline
\\
$\bm{D}_q$ &
$\tiny{\left[
\begin{array}{rrrrrrrrr}
 -2 & -1&q-2&q-3&\hdots&1& 0& 0 &0\\
 -2 & -1&0 & 0&\hdots & 0& 2& 1 &0\\
 -1 &  0&1 & 1 &\hdots & 1& 1& 1&1\\
\end{array}
\right]}$
& \tiny{$T_{11}^2  T_{12}+ T_{21}^ {q-2} \cdots T_{2,q-2}+ T_{31} ^{2}T_{32}$}\\
\\
\hline
\\
$\bm{E}_6$ &
\tiny{$\left[
\begin{array}{rrrrrrrrr}
 -3 & -2&-1&3 &2&1& 0&0&0\\
 -3 & -2&-1&0 &0&0& 2&1&0\\
 -2 & -1& 0&1 &1&1& 1&1&1\\
\end{array}
\right]$}
& \tiny{$T_{11}^3 T_{12}^2T_{13}+ T_{21}^3 T_{22}^2T_{23} +  T_{31}^2 T_{32}$}\\
\\
\hline
\\
$\bm{E}_7$&
$\tiny{\left[
\begin{array}{rrrrrrrrrr}
 -4&-3 & -2&-1&3 &2&1& 0&0&0\\
-4& -3 & -2&-1&0 &0&0& 2&1&0\\
 -3&-2 & -1& 0&1 &1&1& 1&1&1\\
\end{array}
\right]}$
& \tiny{$T_{11}^4\cdots T_{14}+ T_{21}^3 T_{22}^2T_{23} +  T_{31}^2 T_{32}$}\\
\\
\hline
\\
$\bm{E}_8$&
\tiny{$\left[
\begin{array}{rrrrrrrrrrr}
 -5 &  -4&-3 & -2&-1&3 &2&1& 0&0  & 0\\
 -5 &  -4& -3 & -2&-1&0 &0&0& 2&1 & 0\\
 -4 &   -3&-2 & -1& 0&1 &1&1& 1&1 & 1\\
\end{array}
\right]$}
& \tiny{$T_{11}^5\cdots T_{15}+ T_{21}^3 T_{22}^2T_{23} +  T_{31}^2 T_{32}$}
\end{longtable}}
\end{center}
\noindent
The fan $\tilde \Sigma$ of the canonical ambient 
toric variety $\tilde Z$ of $\tilde X$ is the 
unique fan with only two-dimensional maximal cones,
all of them lying on $\trop(X)$,
and having as one-dimensional precisely the rays 
through the columns of $\tilde P$.
The corresponding bunch of cones $\tilde \Phi$ is
the Gale dual of $\tilde \Sigma$.
\end{proposition}

\goodbreak

\begin{proof}
We only have to consider the case that 
$X = X(A,P)$ is not a toric surface and 
thus can assume $r \ge 2 $ and 
$l_i := l_{i1} > 1$ for all $i = 0, \ldots, r$. 
We resolve the singularity $x_0 \in X$ 
according to Construction~\ref{constr:ressing}.
The first step gives us a fan with $r+1$ maximal 
cones, each of dimension two:
$$
\cone(v_0,e_{r+1}), \ldots, \cone(v_r,e_{r+1}),
$$
where $v_i \in \QQ^{r+1}$ denotes the $i$-th column 
of $P$ and we may assume that $e_{r+1}$ is the 
$(r+1)$-th canonical basis vector.
In the second step, we perform the minimal 
regular subdivision of these cones.
This gives indeed a minimal resolution 
$\tilde X \to X$ of $x_0$ and the resulting 
picture reflects the resolution graph.
We see that $x_0$ cannot be of type $\bm{A}$ 
und thus is of type $\bm{D}$ or $\bm{E}$.
We end up with $r=2$ and defining data
$$
A
\ = \ 
\left[
\begin{array}{rrr}
0 & -1 & 1 
\\
1 & -1 & 0
\end{array}
\right],
\qquad
\qquad
P 
\ = \ 
\left[
\begin{array}{rrr}
-l_0 & l_1 & 0
\\
-l_0 & 0   & l_2
\\
d_0  & d_1 & d_2 
\end{array}
\right].
$$
Moreover, because all exceptional curves are of 
self intersection -2, we must have 
$d_i \equiv 1 \mod l_i$ for $i=1,2,3$.
That means, that we inserted $l_i-1$ new 
rays to obtain the minimal regular subdivision of 
the $i$-th cone.

As a sample, we continue the case of an 
$\bm{E}_6$-singularity.
By the shape of the corresponding resolution graph, 
we have $l_0 = l_1 = 3$ and $l_2 = 2$ after 
renumbering the columns suitably.
This establishes the $P_0$-block.
By suitable row operations we can achieve
$$
P 
\ = \ 
P_c 
\ := \ 
\left[
\begin{array}{rrr}
- 3 & 3 & 0
\\
- 3 & 0 & 2
\\
1+3c   & 1 & 1 
\end{array}
\right].
$$
Every $c \in \ZZ$
yields a matrix $P_c$ admissible for 
Construction~\ref{constr:RAPdown}.
The minimal resolution $\tilde X_c$ of $X_c$
is the $\KK^*$-surface defined by $A$ and 
the matrix
$$
\tilde P_c
\ = \ 
\left[
\begin{array}{rrrrrrrrr}
 -3   & -2   & -1  & 3 & 2 & 1 & 0 & 0 & 0 \\
 -3   & -2   & -1  & 0 & 0 & 0 & 2 & 1 & 0 \\
 1+3c & 1+2c & 1+c & 1 & 1 & 1 & 1 & 1 & 1 \\
\end{array}
\right].
$$
Only for $c = -1$, we obtain 
self intersection $-2$ for all 
exceptional curves; in fact, 
the one corresponding to $(0,0,1)$ 
is important here.
The fan $\tilde \Sigma$ of the canonical 
ambient toric variety $\tilde Z$ of 
$\tilde X = \tilde X_{-1}$ sits on the 
tropical variety $\trop(X)$ and looks 
as follows:
\begin{center}
\begin{tikzpicture}[scale=0.4]
\draw[thick, draw=black, fill=gray!30!] 
(0,2.1) -- (-3,2.1) -- (-3,-2) -- (0,-2) -- cycle; 
\draw[thick, draw=black, fill=gray!90!] 
(0,-1) -- (-3,0.91) -- (0,1) -- cycle; 
\draw[thick, draw=black] (0,-1) -- (-2,0.97);
\draw[thick, draw=black] (0,-1) -- (-1,0.94);
\draw[black, fill=black] (-3,0.92) circle (.7ex);
\draw[black, fill=black] (-2,0.94) circle (.7ex);
\draw[black, fill=black] (-1,0.97) circle (.7ex);
\draw[thick, draw=black, fill=gray!10!] 
(0,2.1) -- (2.4,3.1) -- (2.4,-1) -- (0,-2) -- cycle; 
\draw[thick, draw=black, fill=gray!40!] 
(0,-1) -- (2.4,0.9) -- (0,1) -- cycle; 
\draw[thick, draw=black] (0,-1) -- (2.4,0.95);
\draw[black, fill=black] (2.4,0.9) circle (.7ex);
\draw[thick, draw=black] (0,-1) -- (1.2,0.95);
\draw[black, fill=black] (1.2,0.95) circle (.7ex);
\draw[thick, draw=black, fill=gray!20!, fill opacity=0.80] 
(0,2.1) -- (2.65,1.1) -- (2.65,-3) -- (0,-2) -- cycle; 
\draw[thick, draw=black, fill=gray!60!, fill opacity=0.80] 
(0,-1) -- (2.65,-0.2) -- (0,1) -- cycle; 
\draw[thick, draw=black] (0,-1) -- (0.883,0.6);
\draw[thick, draw=black] (0,-1) -- (1.766,0.2);
\draw[black, fill=black] (2.65,-0.2) circle (.7ex);
\draw[black, fill=black] (0.883,0.6) circle (.7ex);
\draw[black, fill=black] (1.766,0.2) circle (.7ex);
\draw[black, fill=black] (0,1) circle (.7ex);
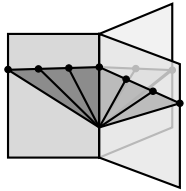
\node at (0,-4) {\tiny Resolution of the $\bm{E}_6$-Singularity};
\end{tikzpicture}
\end{center}
\end{proof}

\begin{remark}
Note that the approach via the defining 
data~$A$ and $P$ establishes 
\emph{a posteriori} that every du Val surface 
singularity can be realized as the fixed point 
of an elliptic $\KK^*$-surface.
Similarly, the defining equation
for $X \subseteq \KK^3$ is easily seen 
to be the defining relation of the Veronese 
subalgebra $\Gamma(X,\mathcal{O}) = R(A,P)_0$ 
of the Cox ring $\mathcal{R}(X) = R(A,P)$.
\end{remark}

\end{document}